\documentclass{amsart}
\usepackage{amsmath,amssymb,amsthm}

\theoremstyle{plain}
\newtheorem{theorem}{Theorem}[section]
\newtheorem{lemma}[theorem]{Lemma}
\newtheorem{thm}[theorem]{Theorem}
\newtheorem{prop}[theorem]{Proposition}

\newtheorem{cor}[theorem]{Corollary}

\newtheorem{obs}[theorem]{Observation}
\newtheorem{rem}[theorem]{Remark}

\newtheorem{porism}[theorem]{Porism}

\theoremstyle{definition}
\newtheorem{definition}[theorem]{Definition}

\def\cvg{\!\downarrow} 

\DeclareMathOperator \dom{dom}

\renewcommand{\le}{\leqslant}
\renewcommand{\leq}{\leqslant}
\renewcommand{\ge}{\geqslant}

\newcommand{\w}{\omega}
\newcommand{\rest}{\upharpoonright}
\newcommand{\vphi}{\varphi}
\newcommand{\s}{\sigma}
\newcommand \tth{{}^{\textup{th}}}

\newcommand \seq[1]{{\left\langle{#1}\right\rangle}}
\newcommand{\degr}[1]{\mathbf{#1}} 
\newcommand {\Tot}{\textup{Tot}}
\DeclareMathOperator \range{range}

\begin{document}

\title[Reducibilities with tiny use]{Anti-complex sets
and reducibilities \\ with tiny use} 

\author[Franklin]{Johanna N.Y. Franklin}
\address{J.N.Y.\ Franklin, Department of Mathematics, 6188 Kemeny Hall,
Dartmouth College, Hanover, NH 03755-3551, USA}
\email{johannaf@gauss.dartmouth.edu}

\author[Greenberg]{Noam Greenberg}
\address{N.\ Greenberg, Victoria University, PO Box 600, Wellington 6140,
New Zealand}
\email{greenberg@mcs.vuw.ac.nz}

\author[Stephan]{Frank Stephan}
\address{F.\ Stephan, Department of Computer Science and Department
of Mathematics, National University of Singapore, Singapore 117543,
Republic of Singapore}
\email{fstephan@comp.nus.edu.sg} 

\author[Wu]{Guohua Wu}
\address{G.\ Wu, School of Physical and Mathematical Sciences, 
Nanyang Technological University, Singapore 639798, Republic of Singapore.} 
\email{guohua@ntu.edu.sg}

\thanks{The second author was supported in part by the Marsden grant
of New Zealand and by NTU grant RG58/06, M52110023. The third author
is supported in
part by the NUS grants R252-000-308-112 and R252-000-420-112; he
worked on this paper while on sabbatical leave to Victoria
University of Wellington in October and November 2010.}

\begin{abstract}
In contrast with the notion of complexity, a set $A$ is called
anti-complex if the Kolmogorov complexity of the initial segments of $A$
chosen by a recursive function is always bounded by the identity
function. We show that, as for complexity, the natural arena for
examining anti-complexity is the weak-truth table degrees. In this
context, we show the equivalence of anti-complexity and other lowness
notions such as r.e.\ traceability or being weak truth-table reducible to
a Schnorr trivial set. A set $A$ is anti-complex if and only if it is
reducible to another set $B$ with \emph{tiny use}, whereby we mean that
the use function for reducing $A$ to $B$ can be made to grow arbitrarily
slowly, as gauged by unbounded nondecreasing recursive functions. This notion
of reducibility is then studied in its own right, and we also investigate 
its range and the range of its uniform counterpart. 
 %  
 % A Turing reduction from $A$ to $B$ is said to have tiny use if the use
 % of computing $A(0)A(1)\ldots A(n)$ from $B$ is growing smaller than every
 % nondecreasing unbounded recursive function. It is shown that a set $A$ is
 % reducible to some other set $B$ with tiny use iff $A$ is anti-complex
 % iff $A$ is weak truth-table reducible to a Schnorr trivial set. Here
 % anti-complex means that for every recursive nondecreasing unbounded
 % function $f$ and almost all $n$ the Kolmogorov complexity of
 % $A(0)A(1)\ldots A(f(n))$ is at most $n$. It turns out that the anti-complex
 % sets together with weak truth-table reducibility are to a certain degree
 % the counterpart of Schnorr-trivial sets with truth-table reducibility.
 % In the present work, the notions of anti-complex sets and Turing
 % reductions with tiny use are investigated in depth and many natural
 % variants, like truth-table and many-one reductions with tiny use,
 % are considered. For the related notions, the tiny use can only
 % be obtained in a non-uniform way.
\end{abstract}

\maketitle

\section{Introduction}

\noindent
In a recent talk \cite{Nies:Nanjing_talk}, Nies gave a general
framework for relating lowness notions and their dual highness notions
with what he names ``weak reducibilities'' (with a prominent example
being $\le_{LR}$, the low-for-randomness partial ordering). Even
before their extensive investigation in the context of effective
randomness, in classical recursion theory, both strong and weak
reducibilities gave rise to lowness and highness classes. For example,
truth-table (or weak truth-table) reducibility induced the classes of
superlow and superhigh Turing degrees; in the other extreme,
hyperarithmetic reducibility (and the hyperjump) allowed the
definition of the useful class of hyperlow hyperdegrees (see
\cite{Sacks:90}). In this paper we give a new notion of relative
strength which, surprisingly, leads to a lowness notion which is
analogous to familiar ones in the context of the weak truth-table degrees. 

The motivation for our notion, which we call ``Turing reducibility
with tiny use'', comes from recent investigations into strengthenings
of weak truth-table reducibility in a direction which is incomparable
with truth-table reducibility, namely \emph{computable Lipshitz}
reducibility $\le_{cl}$ (also known as $\le_{sw}$, \emph{strong weak
truth-table} reducibility) and \emph{identity-bounded} Turing
reducibility $\le_{ibT}$, and also, to a smaller extent, related
reducibilities such as $\le_{C}$ and $\leq_H$. These reducibilities
were introduced in order to combine the traditional Turing reduction,
that is, computation of the membership relation using an oracle, and
calibration of relative randomness, usually on the domain of left-r.e.\
reals (see \cite{DHNT:calibrating}). 

Recall that computable Lipschitz reductions are weak truth-table
reductions in which the use function is bounded by $n+c$ for some
constant $c$. The idea of a Turing reduction with tiny use is to
further restrict the use function of the reduction to recursive
functions which grow more and more slowly. A set $A$ is reducible to a
set $B$ with tiny use if one can use arbitrarily little of the oracle
$B$ to compute arbitrarily much of $A$, so not only does $B$ contain
all the information that $A$ has, it compresses that information
arbitrarily well. To make the definition precise, we invoke the
following definition first made by Schnorr \cite{Schnorr:75}: an
\emph{order function} (or simply an \emph{order}) is a recursive
function which is nondecreasing and unbounded. 

\begin{definition}\label{def:T(tu)}
 Let $A,B\in \{0,1\}^\w$. We say that $A$ is \emph{reducible to $B$ with
tiny use} and write $A\le_{T(tu)} B$ if for every order function
$h$, there is a Turing reduction of $A$ to $B$ whose use function is
bounded by $h$. 
\end{definition}

\noindent
Let us agree on some notation and terminology. If $\Phi$ is a Turing
reduction (a Turing machine with an oracle) and $\Phi^B=A$, then we
let, for every $n<\w$, the \emph{use} of this reduction, $\vphi(n) =
\vphi^B(n)$, be $m+1$, where $m$ is the largest number which is
queried by $\Phi$ while computing $A\rest n$. Here $A\rest n$ is the
string $A(0)A(1)\ldots A(n-1)$, and we assume that before computing
$A(n)=\Phi^B(n)$, $\Phi$ first computes $\Phi^B(m)$ for all $m<n$.
Thus $B\rest \vphi(n)$ is the shortest initial segment of $B$ which,
serving as an oracle for $\Phi$, outputs $A\rest n$.   

The motivation for considering reducibility with tiny use comes from a
result of Greenberg and Nies \cite{GN08}, who showed that if $A$ is a
recursively enumerable, strongly jump-traceable set and $B$ is an
$\w$-r.e.\ random set, then $A$ is reducible to $B$ with tiny use. In
fact, in \cite{GHN} it is shown that this is a characterisation of the
strongly jump-traceable r.e.\ sets. 

The relation $\le_{T(tu)}$ yields a lowness notion in a very simple
way: we consider the collection of sets $A$ for which there is some $B$
such that $A\le_{T(tu)}B$. An immediate analysis of $\le_{T(tu)}$
shows that this collection is invariant in the weak truth-table
degrees and induces an ideal in these degrees. This ideal can be
characterised in three other ways, for which we make a sequence of definitions.

Recall that for their work characterising lowness for Schnorr
randomness as recursive traceability, extending a fundamental result
of Terwijn and Zambella \cite{Terwijn_Zambella}, Kjos-Hanssen, Merkle
and Stephan \cite{Kjos.ea:2005} defined a set $A$ to be \emph{complex}
if there is an order function $f$ such that $C(A\rest f(n))\ge n$ for
all $n$ (here $C$ denotes plain Kolmogorov complexity\footnote{Recall
that a \emph{machine} is a partial recursive function $M\colon
\{0,1\}^*\to \{0,1\}^*$. If $M$ is a machine, then the $M$-complexity of a
string $\s$ in the range of $M$, denoted by $C_M(\s)$, is the length of the
shortest string $\tau\in M^{-1}\{\s\}$. If $\s$ is not in the range of
$M$, then we write $C_M(\s)=\infty$. A machine $U$ is \emph{optimal} if
for every machine $M$ there is some constant $c$ such that for all
$\s\in \{0,1\}^*$, $C_U(\s) \le C_M(\s)+c$. Then $C$ denotes $C_U$ for
some fixed optimal machine $U$.}). They showed that a set $A$ is
complex if and only if there is some diagonally nonrecursive function
$f\le_{wtt} A$. As an analogue, we make the following definition:

\begin{definition}\label{def:anti-complex}
 A set $A\in \{0,1\}^{\w}$ is \emph{anti-complex} if for every order function
$f$, for almost all $n$, $C(A\rest f(n))\le n$.
\end{definition}

\noindent
Thus anti-complexity is a mirror image of complexity: complexity
indicates incompressibility in that one can effectively find
locations of high complexity, whereas anti-complexity denotes a high
level of compressibility and hence low information content. 

Traceability, in both its recursive and r.e.\ versions, is a notion
which has turned out to be extremely useful in algorithmic randomness
and classical recursion theory. Recent work of Franklin and Stephan
\cite{FS08} has indicated that it is also useful in the context of
strong reducibilities. They have shown that the class of Schnorr
trivial sets is invariant in the truth table degrees and that a set
is Schnorr trivial if and only if its truth-table degree is
recursively traceable (this means that only the functions which are
truth-table reducible to the set receive a recursive trace, all with a
uniform bound of some order). Since the natural environment for
$\le_{T(tu)}$ is the weak truth-table degrees, we find that
traceability in these degrees plays a role here. The characterisation
theorem is as follows.

\begin{theorem}\label{thm:lowness_main_thm}
 The following are equivalent for a set $A$.
 \begin{enumerate}
  \item There is a set $B$ such that $A\le_{T(tu)} B$.
  \item $A$ is anti-complex. 
  \item $\deg_{wtt}(A)$ is r.e.\ traceable.
  \item $A$ is weak truth-table reducible to a Schnorr trivial set. 
 \end{enumerate}
\end{theorem}

\noindent
We note that the equivalence of (3) and (4), together with Franklin
and Stephan's result, yields a theorem which has no explicit
connection to effective randomness, and yet we currently do not know
of any direct proof that does not involve $\le_{T(tu)}$ and Kolmogorov
complexity: a weak truth-table degree $\degr{a}$ is r.e.\ traceable if
and only if there is some weak truth-table degree $\degr{b}\ge
\degr{a}$ which contains a set $B$ whose truth-table degree is
recursively traceable. 

As notions of traceability of total functions are equivalent to their
strong versions (see Lemma \ref{lem:order_doesnt_matter}), it follows
that a Turing degree $\mathbf{a}$ is r.e.\ traceable if and only if
every weak truth-table degree contained in $\mathbf{a}$ is r.e.\
traceable. Theorem \ref{thm:lowness_main_thm} then implies the
following characterisation of r.e.\ traceability in the Turing degrees.

\begin{theorem}\label{thm:re traceable Turing degrees}
 The following are equivalent for a Turing degree $\mathbf{a}$.
 \begin{enumerate}
  \item $\mathbf{a}$ is r.e.\ traceable.
  \item Every set $A\in \mathbf{a}$ is anti-complex.
  \item Every set $A\in \mathbf{a}$ is weak truth-table reducible to a
Schnorr trivial set. 
 \end{enumerate}
\end{theorem}

\noindent
Among r.e.\ degrees, we note that the equivalence between array
recursiveness and r.e.\ traceability holds in the weak truth-table
degrees. Recall that a \emph{very strong array} $\bar F =
\seq{F_n}_{n<\w}$ consists of a recursive sequence of pairwise
disjoint finite sets such that for all $n$, $|F_n|\ge n$, and that an
r.e.\ set $A$ is $\bar F$-ANR if for every r.e.\ set $B$ there are
infinitely many $n$ such that $A$ and $B$ coincide on $F_n$. 

\begin{theorem}\label{thm:r.e. degrees}
 The following are equivalent for a weak truth-table degree
$\mathbf{a}$ containing an r.e.\ set.
 \begin{enumerate}
  \item For no very strong array $\bar F$ does $\mathbf{a}$ contains
an $\bar F$-ANR set. 
  \item For some very strong array $\bar F$, $\mathbf{a}$ contains no
$\bar F$-ANR set. 
  \item There is an $\w$-r.e.\ function that dominates all functions recursive 
in $\mathbf{a}$. 
  \item $\mathbf{a}$ is r.e.\ traceable.
  \item For all $A\in \mathbf{a}$, $A\le_{T(tu)}K$ $($here
$K=\{e: \varphi_e(e)\cvg\}$ is the halting set$)$. 
 \end{enumerate}
\end{theorem}

\noindent
This result implies the result from
\cite{Downey_Jockusch_Stob:_array_recursive_1} that the array
recursive r.e.\ wtt-degrees form an ideal. 

\medskip
\noindent
Together with $\le_{T(tu)}$, we also investigate a uniform version
$\le_{uT(tu)}$, where a single reduction witnesses the relation
$\le_{T(tu)}$. This relation is, in general, much stronger than
$\le_{T(tu)}$ (for example, if $A$ is nonrecursive and
$A\le_{uT(tu)}B$, then $B$ is high), but their domains are the same,
and so the condition ``there is a set $B$ such that $A\le_{uT(tu)}B$''
can be added as a fifth equivalent condition in Theorem
\ref{thm:lowness_main_thm}. An even stronger version of this theorem
which bounds the complexity of such $B$ is Theorem
\ref{thm:first_equiv}. We prove Theorem
\ref{thm:lowness_main_thm} in Section \ref{sec:low}. In Section
\ref{sec:distribution} we investigate the distribution of the
anti-complex sets in the Turing degrees, discuss high and random
degrees, prove Theorems \ref{thm:re traceable Turing degrees} and
\ref{thm:r.e. degrees}, and
investigate anti-complexity and tiny use in the r.e.\ degrees. One
corollary of our investigations is an answer to Question 7.5.13 from
Nies's book \cite{Nies:book}. 

\begin{theorem}\label{thm:pcr}
Not every high Turing degree contains a partial-recursively random set. 
\end{theorem}

\noindent
The motivation behind this question is to find an exact boundary
between weaker notions of randomness, such as Schnorr randomness and
recursive randomness, which occur in every high Turing degree, and
stronger notions of randomness, such as Martin-L\"{o}f randomness,
which do not. We provide a proof of Theorem \ref{thm:pcr} in Section
\ref{sec:distribution}.

In Section \ref{sec:high}, we investigate the dual highness notions:
the sets $B$ for which there is a nonrecursive set $A$ such that
$A\le_{T(tu)}B$ (or the more stringent $A\le_{uT(tu)}B$). We
investigate the situation in both the hyperimmune-free
($\degr{0}$-dominated) degrees and in the r.e.\ and $\Delta^0_2$
degrees. For example, we show that every high Turing degree contains
sets $A$ and $B$ such that $A\le_{uT(tu)} B$ and that for every
nonrecursive r.e.\ set $B$ there is some nonrecursive r.e.\ set $A$
such that $A\le_{T(tu)} B$. 

Throughout the paper, we also mention strong reducibilities (such as
truth-table and many-one) with tiny use. In particular, in Theorem
\ref{thm:Schnorr and tt} we use truth-table reducibility with tiny use
to obtain a new characterisation of Schnorr triviality: a set $A$ is
Schnorr trivial if and only if it is truth-table reducible to some set
$B$ with tiny use. This result strengthens the intuition, arising from
Franklin and Stephan's characterisation of Schnorr triviality in terms
of recursive traceability in the truth-table degrees, that strong
reducibilities have deep connections with weak randomness notions.
Along this vein, Day \cite{Day} has recently given characterisations
of both Schnorr randomness and computable randomness as the
complements of the domains of relations weaker than truth-table
reducibility with tiny use. For example, he showed that a set $A$ is
not Schnorr random if and only if there is some set $B$ such that
$A\le_{tt} B$ with use function which does not dominate $n-h(n)$ for
some order function $h$. 

\medskip
\noindent
In the following section we supply the rest of the basic definitions
and make some basic observations. 

\section{Basics}

\noindent
We first define the uniform reducibility.

\begin{definition}\label{def:uT(tu)}
 Let $A,B\in \{0,1\}^\w$. We say that $A$ is \emph{uniformly reducible to
$B$ with tiny use} (and write $A\le_{uT(tu)} B$) if there is a Turing
reduction $\Phi^B=A$ whose use function is dominated by every order function. 
\end{definition}

\begin{obs} \
 \begin{enumerate}
 \item If $A\le_{uT(tu)} B$, then $A\le_{T(tu)} B$. 
 \item If $A \le_{T(tu)} B$, then $A\le_{wtt} B$.
 \end{enumerate}
\end{obs}

\begin{rem}
 Despite the fact that our reductions imply weak truth-table
reductions, we prefer the notation $\le_{T(tu)}$ to $\le_{wtt(tu)}$.
This is because a weak truth-table reduction first marks the use, then
queries the oracle and finally computes the value, whereas Turing
reductions with tiny use would --- at least in the uniform case --- not do
the operations in this order, as otherwise the use is
automatically bounded from below by an order function.
\end{rem}

\noindent
Next, we see that our relations are invariant in the wtt-degrees. 

\begin{obs} \label{obs:wtt_invariance}
 If $A \le_{wtt} E$ and $E\le_{T(tu)} B$, then $A\le_{T(tu)}B$; if
$A\le_{T(tu)}E$ and $E\le_{wtt}B$, then $A\le_{T(tu)}B$. Thus the
relation $\le_{T(tu)}$ is invariant on weak truth-table degrees and
is preserved by increasing the degree on the range and decreasing the
degree on the domain. The same holds for $\le_{uT(tu)}$.
\end{obs}

\begin{obs} \label{obs:ideal}
For a fixed $B\in \{0,1\}^\w$, the classes  $\{A: A \leq_{T(tu)} B\}$ and
$\{A: A \leq_{uT(tu)} B\}$ are wtt-ideals. 
\end{obs}

\noindent
Another formulation for our notions uses not the use functions but
their discrete inverses. If $\Phi^B=A$ is a Turing reduction, then for
every $n<\w$ we let $\Phi(B\rest n)$ be the longest initial segment of
$A$ which is calculated by $\Phi$ querying the oracle $B$ only on
numbers smaller than $n$. 

In general, if $f\colon \w\to \w$ is a nondecreasing and unbounded
function but not necessarily recursive, we let $f^{-1}$, the
\emph{discrete inverse} of $f$, be defined by letting $f^{-1}(k)$ be
the greatest $n$ such that $f(n)\le k$ (let us assume that $f(0)=0$,
as it is for every use function, so $f^{-1}$ is total; otherwise
$f^{-1}$ is defined for almost all numbers). That is, if
$f(n+1)>f(n)$, then the interval $[f(n), f(n+1))$ gets mapped by
$f^{-1}$ to $n$. We note that if $f$ is recursive (and is thus an
order function), then so is $f^{-1}$. 

According to this definition, if $\Phi^B=A$ with use $\vphi$, then for
all $n$, $\Phi(B\rest n) = A\rest \vphi^{-1}(n)$.  

\begin{obs} Let $f$ and $g$ be nondecreasing and unbounded. 
\begin{enumerate}
 \item If $f$ bounds $g$, then $g^{-1}$ bounds $f^{-1}$.
 \item $\left(f^{-1}\right)^{-1}$ bounds $f$. If $f$ is growing slower
than the identity function, that is,
if for all $n$, $f(n+1)\le f(n)+1$, then  $\left(f^{-1}\right)^{-1}=f$. 
\end{enumerate}
\end{obs}

\noindent
This observation suffices for the following corollary, noting that
when investigating slow-growing recursive orders, we may assume that
the orders grow slower than the identity function.

\begin{cor}\label{cor:alternative_def}
 Let $A,B\in \{0,1\}^\w$.
 \begin{enumerate}
  \item $A\le_{T(tu)}B$ if and only if for every order function $g$,
there is a Turing reduction $\Phi^B=A$ such that the map $n\mapsto
|\Phi(B\rest n)|$ bounds $g$. 
  \item $A\le _{uT(tu)} B$ if and only if there is a Turing reduction
$\Phi^B=A$ such that the map $n\mapsto |\Phi(B\rest n)|$ dominates
every recursive function. (A function which dominates every recursive
function is called \emph{dominant}.)
 \end{enumerate}
\end{cor}

\noindent
Some other basic results follow. 

\begin{prop}  Let $A,B\in \{0,1\}^\w$.
 \begin{enumerate}
 \item If $A\le_{T(tu)}A$, then $A$ is recursive.
 \item If $A$ is recursive, then $A\le_{uT(tu)}B$.
 \end{enumerate}
\end{prop}

\begin{proof}
 Let $f(n)=n+1$. If $\Phi^A=A$ and for all $n$ we have $\Phi(A\rest
n)\supseteq A\rest n+1$, then we can recursively compute $A(n)$ by
applying $\Phi$ to $A\rest n$, which we already computed. For (2), use
a reduction $\Phi^B=A$ whose use function is a constant 0. 
\end{proof}

\begin{cor}\label{cor:minimal_wtt}
 If $A\le_{T(tu)} B$ and $A$ is nonrecursive, then $\deg_{wtt}(A) <
\deg_{wtt}(B)$.
\end{cor}

\noindent
As a result, if $\deg_{wtt}(B)$ is minimal, then every $A\le_{T(tu)}B$
is recursive. 

\begin{prop} \label{prop: high1}
 Let $B\in \{0,1\}^\w$. If there is some nonrecursive $A$ such that
$A\le_{uT(tu)}B$, then $B$ is high. 
\end{prop}

\noindent
Recall that a set $B$ is high if $B'\ge_T \emptyset''$.

\begin{proof}
 For any Turing reduction, if $\Phi^B$ is total, then the map
$n\mapsto |\Phi(B\rest n)|$ is computable in $B$ (indeed, weak
truth-table reducible to $B$). The map $\Phi$ which witnesses
$A\le_{uT(tu)}B$ dominates every recursive function. By Martin
\cite{Martin:66}, this map has high Turing degree. 
\end{proof}

\noindent
We will review the situation in Proposition \ref{prop: high1} in
greater detail in Section \ref{sec:high}.

\section{Sets bounded by other sets with tiny use} \label{sec:low}

\noindent
In this section we prove Theorem \ref{thm:lowness_main_thm}. It will
follow from Theorem \ref{thm:first_equiv} and Propositions
\ref{prop:r.e.traceability_anticomplexity},
\ref{prop:trivial_is_anticomplex} and \ref{prop:schnorr_triviality}.

\subsection{Anti-complexity and traceability}

For functions $f,g\colon \w\to \w$, we write $f \le^+ g$ if there is
some constant $c$ such that $g+c$ bounds $f$. 

\begin{lemma}\label{lem:anticomplex_C(f)}
 A set $A$ is anti-complex if and only if for every $f\le_{wtt}A$, \[
C(f(n))\le^+ n.\]
\end{lemma}

\noindent
This lemma shows that the notion of anti-complexity (like its analogue
notion, complexity) is wtt-degree invariant. 

\begin{proof}
 We first note that $A$ is anti-complex if and only if for every order
function $f$, $C(A\rest f(n))\le^+ n$. One direction is immediate from
Definition \ref{def:anti-complex}. For the other direction,
suppose that for every order function $f$, $C(A\rest f(n))\le^+ n$.
Let $f$ be an order function. Applying the hypothesis twice to the
functions $n\mapsto 2n$ and $n\mapsto 2n+1$, there is a constant $c$
such that for all $n$, $C(A\rest f(2n))\le n+c$ and $C(A\rest
f(2n+1))\le n+c$. If $n\ge c$, then $C(A\rest f(2n))$ and $C(A\rest
f(2n+1))$ are less than or equal to $2n$, so Definition
\ref{def:anti-complex} holds.  

 Assume that for every $g\le_{wtt} A$, $C(g(n))\le^+ n$. Let $f$ be an
order function and let $g(n)$ be a natural number code for $A\rest
f(n)$. Then $g\le_{wtt}A$, so as we just observed, $A$ is anti-complex. 

 Now assume that $A$ is anti-complex and let $h$ be an order
function. Let $f\le_{wtt} A$ and let $g$ be a recursive bound for the
use function for the reduction of $f$ to $A$. Using this reduction, we
see that $C(f(n)) \le^+ C(A\rest g(n))$. Again, as we just observed,
$C(A\rest g(n))\le^+ n$.
\end{proof}

\noindent
We show that anti-complexity can also be characterised as a weak
truth-table analogue of a very useful concept in the Turing degrees,
that of r.e.\ traceability. Recall that a \emph{r.e.\ trace} for a
function $f$ is a uniformly recursively enumerable sequence
$\seq{T_n}$ of finite sets such that for all $n$, $f(n)\in T_n$, and
that a trace $\seq{T_n}$ is \emph{bounded} by an order function $h$ if
the function $n\mapsto |T_n|$ is bounded by $h$. 

\begin{definition}\label{def:wtt_ce_traceable}
 A weak truth-table degree $\degr{a}\in \mathcal D_{wtt}$ is
\emph{r.e.\ traceable} if there is an order function $h$ such that
every $f\le_{wtt} \degr{a}$ has an r.e.\ trace which is bounded by $h$.
\end{definition}

\noindent
The standard argument of Terwijn and Zambella \cite{Terwijn_Zambella}
shows that the choice of order doesn't matter:

\begin{lemma}\label{lem:order_doesnt_matter}
 A weak truth-table degree $\degr{a}$ is r.e.\ traceable if and only
if for every order function $h$, every $f\le_{wtt} \degr{a}$ has an
r.e.\ trace which is bounded by $h$.
\end{lemma}

\begin{proof}
 Suppose that $h$ is an order function which witnesses that a weak
truth-table degree $\degr{a}$ is r.e.\ traceable. Let $\hat h$ be any
other order function and let $f\le_{wtt} \degr{a}$. 

 Let $g(n)$ be the least $k$ such that $\hat h(k) \ge h(n)$. This
function is well defined because $\hat h$ is unbounded and is
recursive. Hence the map $n\mapsto f\rest (g(n+1))$ is weak
truth-table below $\degr{a}$, and so it has a trace $\seq{T_n}$ which
is bounded by $h$. 

 The function $g$ is unbounded because $h$ is unbounded. Let $g^{-1}$
be the discrete inverse of $g$, so $g^{-1}(k)$ is the greatest $n$
such that $h(n) \le \hat h(k)$ (note that $g^{-1}$ is defined on
almost every number). Then $|T_{g^{-1}(k)}| \le \hat h(k)$ and
$g(g^{-1}(k)+1)>k$, so $f\rest l$ is an element of $T_{g^{-1}(k)}$ for
some $l>k$. Hence we can let $S_k$ be the collection of all values
$\s(k)$ for all strings of length greater than $k$ that contain only
numbers which
are elements of $T_{g^{-1}(k)}$. Then $\seq{S_n}$ will be an r.e.\
trace for $f$ which is bounded by $\hat h$.
 \end{proof}

\begin{prop}\label{prop:r.e.traceability_anticomplexity}
 A set $A$ is anti-complex if and only if $\deg_{wtt}(A)$ is r.e.\ traceable. 
\end{prop}

\begin{proof}
 Suppose that $A$ is anti-complex and let $f\le_{wtt}A$. By Lemma
\ref{lem:anticomplex_C(f)}, there is some constant $c$ such that for
all $n$, $C(f(n))\le n+c$. Then letting 
 \[ T_n = \{ y\,:\, C(y)\le n+c \},\]
$\seq{T_n}$ is an r.e.\ trace for $f$ and for all $n$, $|T_n| \le
2^{n+c+1}$. Hence (by changing finitely many entries for every
function), $\deg_{wtt}(A)$ is r.e.\ traceable, witnessed by the order
function $h(n) = 2^{2n}$.

 The other direction follows an idea of Kummer's, who showed that
every array recursive r.e.\ Turing degree contains only sets of low
complexity \cite{Kummer:r.e.96} (see also \cite{DG08}). Suppose that
$\deg_{wtt}(A)$ is r.e.\ traceable and let $f\le_{wtt} A$. By Lemma
\ref{lem:order_doesnt_matter}, let $\seq{T_n}$ be an r.e.\ trace for
$f$ which is bounded by the order function $h(n) = n$. We can
construct a machine $M$ which on input $\s$, first computes $U(\s)$,
interprets the result as a pair $(n,m)$ and, if $m<n$, outputs the
$m\tth$ element enumerated into $T_n$. Then for all $n$, if $f(n)$ is
the $m\tth$ element
  enumerated into $T_n$, then $M$ shows that $C(f(n))\le^+ C(n,m)$. 

 Now the standard coding of pairs as numbers is polynomial; so there
is some constant $c$ such that for all $n$ and all $m\le n$,
$\seq{n,m}\le n^c$. For all $x$, the identity machine witnesses that
$C(x)\le^+ \log_2 x$. Hence for all $n$ and all $m\le n$, 
 \[ C(n,m)\le^+ \log_2 (\seq{n,m}) \le \log_2 n^c = c\log_2 n \le^+ n.\]
  Thus we see that the condition of Lemma
\ref{lem:anticomplex_C(f)} holds. 
\end{proof}

\begin{porism} \label{por:log}
 If $A$ is anti-complex, then there is some $c<\w$ such that for all
$f\le_{wtt} A$, $C(f(n))\le^+ c\log_2 n$. (In fact, by working a bit
harder, we can have $c=2$.)
\end{porism}

\subsection{Tiny use}
Given $A\in \{0,1\}^\w$, the function $n\mapsto C(A\rest n)$ is far from
monotone. Nevertheless, we are interested in some form of inverse,
which is possible because $\lim_n C(A\rest n)=\infty$. We let $g_A(k)$
be the least $n$ such that for all $m\ge n$, \linebreak $C(A\rest m)> k$. 

\begin{obs}
 For all $A\in \{0,1\}^\w$, $g_A \le_{T} A\oplus K$. As before,
$K=\emptyset'$ is
the halting problem. 
\end{obs}

\noindent
For any string $x$, we let $x^*$ be the least element of $U^{-1}\{x\}$
(where $U$ is the universal machine we use for plain complexity), so
$C(x)=|x^*|$. We also let 
\[ A^* = \left\{ \left(A\rest g^A(k)\right)^*\,:\, k<\w\right\} .\]
Once again, we get $A^* \le_T A\oplus K$. 

\begin{lemma}\label{lem:A*_dominated}
For every $A\in \{0,1\}^\w$, the map $k\mapsto \left( A\rest g_A(k)
\right)^*$ is bounded by some recursive function. 
\end{lemma}

\begin{proof}
 There is a constant $c$ such that for all $\tau \in \{0,1\}^{*}$,
$C(\tau0)$ and $C(\tau1)$ are both less than or equal to $C(\tau)+c$
(consider the machine which on input
$\s i$, for $i=0,1$, outputs $U(\s)i$). 

For any $k<\w$, let $\tau_k$ be a binary string and $i\in\{0,1\}$ be such
that $A\rest g_a(k) = \tau_k i$. By the definition of $g_A(k)$,
$C(\tau_k)\le k$, and so $C(A\rest g_A(k)) \le k+c$. Hence $\left(
A\rest g_A(k) \right)^* < 2^{k+c+1}$. 

To ensure the last inequality, we need some agreement about the
coding of strings by numbers. This coding is obtained by some
$\w$-ordering of all binary strings; we order binary strings by length
first. We let $|x|$ denote the length of the string identified with
the number $x$, so for all $x$,  $2^{|x|} \le x < 2^{|x|+1}$.
\end{proof} 

\begin{theorem}\label{thm:first_equiv}
 The following are equivalent for $A\in \{0,1\}^\w$.
 \begin{enumerate}
  \item There is some set $B$ such that $A\le_{T(tu)} B$.
  \item $A$ is anti-complex. 
  \item $g_A$ is dominant.
  \item $A\le_{uT(tu)} A^*$.
 \end{enumerate}
\end{theorem}

\noindent
We remark that we are not aware of a shorter proof of the equivalence
of (2) and (3). This suggests that the study of the relation
$\le_{T(tu)}$ is important for the seemingly independent study of
anti-complexity in the wtt-degrees. 

\begin{proof}
 (1) implies (2): 
 Assume that $A\le_{T(tu)} B$. For any functional $\Phi$ such that
$\Phi^B=A$, for all $n$, $C(\Phi(B\rest n))\le^+ C(B\rest n)$. Also,
for all $x$, $C(x)\le^+ |x|$, so for all $n$, $C(\Phi(B\rest n))\le^+
n$. Suppose that $f\le_{wtt} A$, so there is some functional $\Gamma$
such that $\Gamma^A=f$ and the use of this computation is bounded by a
recursive function $g$. We can find some $\Phi$ such that for all $n$,
$\Phi(B\rest n)$ is longer than $A\rest g(n)$, so $C(f(n))\le^+ n$. By
Lemma \ref{lem:anticomplex_C(f)}, $A$ is anti-complex. 

(2) implies (3):  
 Suppose that $A$ is anti-complex and let $f$ be an increasing
recursive function. By definition, for almost all $n$, $C(A\rest
f(n))\le n$. Hence, for almost all $n$, $g_A(n) > f(n)$. It follows
that $g_A$ dominates every recursive function. 

 (3) implies (4):  
For every $A\in \{0,1\}^\w$ we have $A\le_T A^*$ because
\[ A = \bigcup \{ U(\s)\,:\, \s\in A^* \} \] (in other words,
$A(x)=U(\s)(x)$ for any $\s\in A^*$ such that $x<|U(\s)|$, and for
every $x$ there is indeed some $\s\in A^*$ such that $|U(\s)|>x$). 

If $g_A$ is dominant, then this reduction witnesses that $A
\le_{uT(tu)} A^*$. To see this, let $\Phi$ code the described
reduction and let $f$ be an increasing recursive function; we see
that $n\mapsto |\Phi(A^*\rest n)|$ dominates $f$. 

Let $g$ be a recursive function which dominates $k\mapsto \left(A\rest
g_A(k)\right)^*$ (Lemma \ref{lem:A*_dominated}). Since $g_A$ is
dominant, for almost all $k$, $g_A(k) > f(g(k+1))$. Suppose that $k$
is large enough that $\left(A\rest g_A(k)\right)^* < n \le
\left(A\rest g_A(k+1)\right)^*$. Then $x\le g(k+1)$ and so $g_A(k) >
f(n)$. Then $|\Phi(A^*\rest n)| \ge g_A(k)$ so $|\Phi(A^*\rest n)| \ge
f(n)$ as required. 

(4) implies (1):
    This is clear from the definitions.
\end{proof}

\subsection{Schnorr triviality}
Franklin and Stephan \cite{FS08} characterise the Schnorr trivial
sets (defined by Downey and Griffiths in \cite{Downey.Griffiths:04})
as those sets whose \emph{truth-table} degree is \emph{recursively
traceable}, that is, there is some order function $h$ which bounds
traces for all functions $f$ truth-table reducible to the degree
$\degr{a}$, but where the trace $\seq{T_n}$ is required to be given
recursively (as a sequence of finite sets) rather than merely
uniformly recursively enumerably. In other words, there is a recursive
function $g$ such that for all $n$, $g(n)$ is the canonical index for
the finite set $T_n$ (in Soare's \cite{Soare:87} notation, $T_n =
D_{g(n)}$). Again, the Terwijn-Zambella argument shows that any order would do. 

Schnorr triviality is not invariant in the weak truth-table degrees
\cite[Theorem 4.2]{FS08}. However, the downward closure of the
wtt-degrees containing Schnorr trivial sets is familiar. 

\begin{prop}\label{prop:trivial_is_anticomplex}
 Every Schnorr trivial set is anti-complex.  
\end{prop}

\begin{proof}
 Let $A$ be Schnorr trivial. Fix an order function $h$. Let $\Phi$ be
a weak truth-table functional with a recursive bound $g$ on the use
function of $\Phi$. Since the map $n\mapsto A\rest g(n)$ is
truth-table reducible to $A$, by the characterisation mentioned above,
there is a recursive trace $\seq{T_n}$ for this map which is bounded
by $h$. If $\Phi^A$ is total, then we can enumerate a trace $S_n$ for
$f$ with bound $h$ by outputting $\Phi^\s(n)$ for those $\s\in T_n$
for which $\Phi^\s$ converges with domain greater than $n$. Hence
$\deg_{wtt}(A)$ is r.e.\ traceable; by Proposition
\ref{prop:r.e.traceability_anticomplexity}, $A$ is anti-complex. 
\end{proof} 

\begin{prop}\label{prop:schnorr_triviality}
 Let $A\in \{0,1\}^\w$. If $g_A$ is dominant, then $A$ is weak truth-table
reducible to some Schnorr trivial set. 
\end{prop}

\begin{proof}
 Let $f_0,f_1,\dots$ be a sequence of (total) recursive functions such that
 \begin{itemize}
  \item each $f_i$ is strictly increasing,
  \item the range of $f_{i+1}$ is contained in the range of $f_i$, and
  \item every recursive function is bounded by some $f_i$.
 \end{itemize}
(Note that the halting problem $K$ can compute such a sequence.)

By Lemma \ref{lem:A*_dominated}, let $g$ be a recursive function which
bounds the function $k \mapsto  \left( A\rest g_A(k)\right)^*$. 

For each $k>0$, let $q_{k} = \seq{\left( A\rest g_A(k)\right)^*,
f_i(k)}$, where $i$ is the greatest number such that $\seq{g(k), f_i(k)} \le
g_A(k-1)$. Then for all $k>0$, $q_k \le g_A(k-1)$.

Let $B = \{ q_k\,:\, k>0\}$. We claim that $B$ is Schnorr trivial and
that $A\le_{wtt} B$.

To see the latter, let $n<\w$. Let $k = g_A^{-1}(n)$ (that is, the
greatest $k$ such that $g_A(k) \le n$). Then $q_{k+1} \le g_A(k) \le
n$ and $A\rest g_A(k+1)$ can be effectively obtained from $q_{k+1}$.
This procedure allows us to generate $A\rest n$ effectively from
$B\rest (n+1)$. 

To see that $B$ is Schnorr trivial, we appeal to the characterisation
mentioned above. Here is where we use the fact that $g_A$ is dominant.
The point is that for every $i$, all but finitely many elements of $B$
are pairs whose second coordinate is contained in the range of $f_i$.
This is because the map $k\mapsto \seq{g(k), f_{i}(k)}$ is recursive
and thus dominated by $g_A$, so for all but finitely many $k$ we will
have $q_k = \seq{\left( A\rest g_A(k)\right)^*, f_{i'}(k)}$ for some
$i'\ge i$, and the range of $f_{i'}$ is contained in the range of $f_i$. 

Now let $\Psi$ be a truth-table functional; there is some $i$ such
that $f_i$ bounds the use function of $\Psi$. After specifying a fixed
initial segment of $B$ (specifying those $q_{k'}$ whose second
coordinate is not in the range of $f_i$), there are at most
$2^{kg(k)}$ many possibilities for $B\rest f_i(k)$ because, apart from
the finitely many fixed elements, there are only $kg(k)$ many numbers
below $f_i(k)$ which can be elements of $B$, as they all have the form
$\seq{p,f_i(m)}$ for some $p< g(k)$ and $m<k$. After applying $\Psi$,
we get a recursive trace for $\Psi(B)$ whose $k\tth$ element has size
at most $2^{kg(k)}$. Hence $\deg_{tt}(B)$ is recursively traceable (in
the tt-degrees), so as quoted above, $B$ is Schnorr trivial. 
\end{proof}

\subsection{Truth-table reductions with tiny use}
Another connection between tiny use and Schnorr triviality is obtained
by examining truth-table reducibility. Recall that $A\le_{tt}B$ if and
only if there is a Turing reduction $\Phi$ for which $\Phi^X$ is total
for all $X$ and $\Phi^B=A$. We say that $A\le_{tt(tu)}B$ if for every
order function $h$ there is such a functional whose use function is
bounded by $h$. Equivalently, for every order function $h$, there is a
truth-table reduction of $A$ to $B$ for which the size of the $n\tth$
truth table is bounded by $h(n)$. This notion is invariant in the
truth-table degrees. 

Since the use function for a total Turing functional is recursive
(equivalently, the size of the $n\tth$ truth-table of a tt-reduction
is recursive), there is no uniform notion in this context.

The class of all $A$ such that there is a $B$ with $A \le_{tt(tu)} B$
gives us a new characterisation of the Schnorr trivial sets.

\begin{theorem} \label{thm:Schnorr and tt}
  Let $A\in \{0,1\}^\w$. There is a set $B$ such that $A\le_{tt(tu)} B$ if
and only if $A$ is Schnorr trivial.
\end{theorem} 

\begin{proof}
 We begin by assuming that $A\le_{tt(tu)} B$. Let $h$ be an order
function. There is a total reduction $\Phi$ such that $\Phi^B=A$ whose
use function is bounded by $n\mapsto\log (h(n))$. Then a recursive
trace for $n\mapsto A\rest n$ with bound $h$ can be obtained by
applying $\Phi$. Hence $\deg_{tt}(A)$ is recursively traceable.

 Now suppose that $A$ is Schnorr trivial. Again the point is that
$\deg_{tt}(A)$ is recursively traceable, so for any recursive function
$f$, the function $A\mapsto A\rest{f(n)}$ has a recursive trace
bounded by the identity function. 

 Let $\seq{f_i}$ be an enumeration of all increasing total recursive
functions. For each $i<\w$, let $\seq{D^i_n}_{n<\w}$ be a recursive
trace for the function $A\mapsto A\rest{f_i(n)}$ such that for all
$n$, $|D^i_n| = n$.

  We let $B$ be the collection of triples $(i,n,m)$ such that
$A\rest{f_i(n)}$ is the $m\tth$ element of $D^i_n$. 

 Let $i<\w$. Let $\Phi_i$ be the following truth-table functional:
given an oracle $X$ and input $x\in [f_i(n-1), f_i(n))$, find the least
$m\le n$ such that $(i,n,m)\in X$; if the $m\tth$ element of $D^i_n$
is a string $\s$ of length $f_i(n)$, output $\s(x)$. If not, or if
there is no $m\le n$ such that $(i,n,m)\in X$, output $0$. It is clear
that for all $i<\w$, $\Phi_i^B = A$. 

 The standard coding of triples of natural numbers by natural numbers
grows polynomially. Hence, if $g$ is, say, an exponentially growing
recursive function, then for almost all $i$, for all $n$ and $m\le n$,
$(i,n,m) < g(n)$. Hence for almost all $i$, $|\Phi_i(B\rest{g(n)})|
\ge f_i(n)$, whence the function $n\mapsto |\Phi_i(B\rest n)|$
dominates $f_i\circ g^{-1}$. Of course every recursive function is
dominated by some $f_i\circ g^{-1}$, so $A\le_{tt(tu)} B$. 
\end{proof}

\section{The distribution of anti-complex sets} \label{sec:distribution}

\noindent
In this section we investigate how the anti-complex sets are
distributed in the Turing degrees and among certain classes of sets.
Three question are natural:
\begin{itemize}
 \item Which Turing degrees contain anti-complex sets?
 \item Which Turing degrees contain only anti-complex sets?
 \item What kind of sets can be anti-complex?
\end{itemize}
The answer to the second question was mentioned in the introduction:

\begin{prop}\label{prop:totally anti-complex}
 A Turing degree $\mathbf{a}$ contains only anti-complex sets if and
only if $\mathbf{a}$ is r.e.\ traceable. 
\end{prop}

\begin{proof}
 A Turing degree $\mathbf{a}$ is r.e.\ traceable if and only if for
every order function $h$, every $f\in \mathbf{a}$ has an r.e.\ trace
bounded by $h$. Since a Turing degree $\mathbf{a}$ is the union of the
weak truth-table degrees contained in $\mathbf{a}$, by Lemma
\ref{lem:order_doesnt_matter}, a Turing degree $\mathbf{a}$ is r.e.\
traceable if and only if every weak truth-table contained in
$\mathbf{a}$ is r.e.\ traceable. The result now follows from Theorem
\ref{thm:lowness_main_thm}. 
\end{proof}

\noindent
Theorem \ref{thm:re traceable Turing degrees} now follows from Theorem
\ref{thm:lowness_main_thm}.

\subsection{High and random anti-complex sets}
Franklin \cite{F08} shows that every high degree
contains a Schnorr trivial set. It follows from Proposition
\ref{prop:trivial_is_anticomplex} that every high degree contains an
anti-complex set. We improve this result in Corollary \ref{cor:high_1}.

Nies \cite{Nies:Little} constructed a $\Delta^0_2$ perfect tree,
all of whose branches are jump-traceable and thus have r.e.\ traceable
Turing degree. Every perfect $\Delta^0_2$ tree contains a high
path, and so there is a high r.e.\ traceable Turing degree. It follows
from Proposition \ref{prop:totally anti-complex} that there is a high
Turing degree that has only anti-complex elements. Note that such
a high degree cannot be $\Delta^0_2$, as every r.e.\ traceable Turing
degree is GL$_2$.

Now every high degree contains Schnorr random and recursively random
sets \cite{NST}. Hence there is a recursively random, anti-complex
set. On the other hand, sufficient randomness precludes
anti-complexity: Ku\v cera \cite{Kucera:85} has shown that every
Martin-L\"of random set weak truth-table computes a diagonally
nonrecursive function, so every Martin-L\"of random set is complex
and thus certainly not anti-complex. 

This result can be strengthened to show that partial-recursively
random sets are not anti-complex. 

\begin{proof}[Proof of Theorem \ref{thm:pcr}]

Let $A$ be an anti-complex set. By Porism \ref{por:log}, there is some
constant $c<\w$ such that $C(A\rest n)\le^+ c\log_2 n$; so for almost
all $n$, $C(A\rest n)\le (c+1)\log_2 n$. Hence Theorem 7 of
\cite{Merkle} shows that no Mises-Wald-Church stochastic set is
anti-complex. Every partial-recursively random set is
Mises-Wald-Church stochastic (see Section 7.4 of \cite{downeybook}),
and so no partial-recursively random set is anti-complex. As we just
discussed, there is a high Turing degree all of whose elements are
anti-complex, and so such a degree cannot contain a
partial-recursively random set.  
\end{proof}

\subsection{Anti-complex-free Turing degrees}
Not every Turing degree contains anti-complex sets. In fact, we can
find a counterexample within the r.e.\ degrees. Note that this
counterexample cannot be very low, as all array recursive (and hence
superlow) r.e.\ degrees are r.e.\ traceable and cannot be high. 

This result extends the result of Downey, Griffiths and LaForte
\cite{DGL} that there is an r.e.\ degree that contains no Schnorr
trivial sets and utilizes their techniques. 

These techniques involve prefix-free complexity. Recall that a machine
$M$ is \emph{prefix-free} if its domain is an antichain of $\{0,1\}^{*}$,
that is, for all distinct $\s,\tau\in \dom M$, $\s$ is not an initial
segment of $\tau$. There is a prefix-free machine, optimal among all
prefix-free machines, and so \emph{prefix-free Kolmogorov complexity},
which is often denoted by $K$, but which we denote by $H$ (to
differentiate from the halting set $K=\emptyset'$), equals $C_V$ for
some optimal prefix-free machine $V$.   

\begin{lemma}\label{lem:anti complex and prefix-free}
 If $A\in \{0,1\}^\w$ is anti-complex, then for every order function $f$,
$H(A\rest{f(n)})\le^+ n$. 
\end{lemma}

\begin{proof}
We follow the proof of Proposition
\ref{prop:r.e.traceability_anticomplexity}. If $A$ is anti-complex and
$f$ is an order function, then since $\deg_{wtt}(A)$ is r.e.\
traceable, there is an r.e.\ trace $\seq{T_n}$, bounded by the
identity function, for the function $n\mapsto A\rest{f(n)}$. The same
argument in the proof of Proposition
\ref{prop:r.e.traceability_anticomplexity} shows that for all $n$
there is some $m\le n$ such that
 \[ H(A\rest{f(n)}) \le^+ H(m,n).\]
 It is no longer true that $H(x)\le^+ \log_2 x$, but even a crude
bound such as $H(n)\le^+ 2\log_n$ would do to show that for some
constant $c$ we have $H(m,n)\le^+c\log_2 n \le^+ n$ as required. 
\end{proof}

\begin{thm}
 There is an r.e.\ Turing degree that contains no anti-complex sets.
\end{thm}

\begin{proof}
 For any prefix-free subset $D$ of $\{0,1\}^{*}$, we let 
 \[ \mu(D) = \sum_{\tau\in D} 2^{-|\tau|} \]
 be the measure of the subset of the Cantor space defined by $D$ by taking
all infinite extensions of elements of $D$. 

 Theorem 9 of \cite{DGL} states that there is an r.e.\ set $A$ such
that for all $B\equiv_T A$ there is a prefix-free machine $M$ such
that $\mu(\dom(M))$ is a recursive real and such that for infinitely
many $m$, $H(B\rest m)\ge C_M(m)$. 

 The r.e.\ degree we seek is the Turing degree of $A$. Let $B\equiv_T
A$; we show that $B$ is not anti-complex. Let $M$ be a machine for
$B$ as described in the previous paragraph. 

 We first note that $\dom(M)$ is a recursive subset of $\{0,1\}^{*}$: If
$\seq{M_s}$ is a some recursive enumeration of $M$, then $\dom(M)\rest
\{0,1\}^{\le n} = \dom(M_s)\rest \{0,1\}^{\le n}$ for any stage $s$ such that
$\mu(\dom(M)) - \mu(\dom (M_s)) < 2^{-n}$; such a stage $s$ can be
found effectively from $n$. Now the range of $M$ may not be recursive,
but $C_M \rest \range(M)$ is a partial recursive function. 

 We can compute a strictly increasing recursive function $f$ such that
for all $n$, 
 \[ \sum_{\substack{m>f(n) \\ m\in \range M}} 2^{-C_M(m)} \le 2^{-3n} \]
by finding some $s(n)$ such that $\mu(\dom(M))- \mu(\dom(M_{s(n)}))\le
2^{-3n}$ and letting $f(n)$ be greater than any number in the range of
$M_{s(n)}$.  
 Let 
 \[ L = \left\{ (C_M(m)-2f^{-1}(m),m)\,:\, m\in \range M \right\}.\]
 The set $L$ is recursively enumerable. Recall that for any set
$D\subseteq \w^2$, the \emph{weight} $\texttt{wt}(D)$ of $D$ is
$\sum_{(n,m)\in D}2^{-n}$. 
 We have 
 \begin{eqnarray*}
  \texttt{wt}(L) = \sum_{m\in \range M} 2^{2f^{-1}(m)-C_M(m)} =
\sum_{n}2^{2n} \sum_{\substack{m\in \range M\\ m\in [f(n), f(n+1))}}
2^{-C_M(m)} \le  \\
  \sum_n 2^{2n}\sum_{\substack{m\in \range M \\m\ge f(n)}} 2^{-C_M(m)}
\le \sum_n 2^{2n}2^{-3n} = \sum_n 2^{-n} < \infty.
  \end{eqnarray*}
 The Kraft-Chaitin Theorem (see \cite{downeybook,Nies:book}) now
ensures that for all $m$, 
 \[ H(m) \le^+ C_M(m) - 2f^{-1}(m)\]
 (recall that for $m\notin\range M$, we let $C_M(m)= \infty$).

 Suppose that $B$ is anti-complex. Then by Lemma \ref{lem:anti
complex and prefix-free}, $H(B\rest{f(n)})\le^+ n$. Let $m<\w$ and let
$n= f^{-1}(m)$. We can uniformly compute $A\rest m$ if we are given
both $m$ and $B\rest {f(n+1)}$. Since $H$ measures prefix-free
complexity, we have $H(B\rest m) \le^+ H(m) + H(B\rest f(n+1))$ (a
description for $B\rest m$ is a description for $m$ concatenated with
a description for $B\rest f(n+1)$). Overall we get, for all $m$, 
 \[ H(B\rest m)\le^+ H(m) + f^{-1}(m) \le^+ C_M(m) - f^{-1}(m).\]
 Since $f$ is increasing, $f^{-1}$ is unbounded, which would make
it impossible to have infinitely many $m\in \range M$ such that
$H(B\rest m) \ge C_M(m)$. Hence $B$ cannot be anti-complex. 
\end{proof}

\subsection{Anti-complex r.e.\ and $\w$-r.e.\ sets}
The results so far show that if $A$ is anti-complex, then there is
some set $B\le_T A\oplus K$ such that $A\le_{uT(tu)}B$. In general, as
we will see shortly, one cannot improve this to $B\le_{wtt} A\oplus
K$. However, if $A$ is r.e., then we get an improved bound as follows.

\begin{prop}\label{prop: r.e. anticomplex}
 If $A$ is an anti-complex r.e.\ set, then $A\le_{uT(tu)} K$.
\end{prop}

\begin{proof}
 We claim that if $A$ is r.e., then $A^* \le_{wtt} K$; the rest
follows from Theorem \ref{thm:first_equiv}. Fix a recursive
enumeration $\seq{A_s}$ of $A$ and let, at stage $s$, $g_s(k)$ be the
least number $n$ such that no initial segment of $A_s$ of length at
least $n$ has a $U$-description of length at most $k$. Then $g_s$
converges to $g_A$ and is an $\w$-r.e.\ approximation of $g_A$. We can
have $g_{s+1}(k)\ne g_s(k)$ only in three cases:
 \begin{itemize}
  \item there is some $\s\in \dom U_{s+1}\setminus \dom U_s$ of length
at most $k$ and $U(\s)\subset A_{s+1}$;
  \item there is some $\s\in \dom U_{s}$ of length at most $k$ such
that $U(\s)\not\subset A_s$ but $U(\s)\subset A_{s+1}$; or
  \item there is some $\s\in \dom U_{s}$ of length at most $k$ such
that $U(\s)\subset A_s$ but $U(\s)\not\subset A_{s+1}$.
 \end{itemize}
 For each $\s$, each case can happen at most once, and the first two
cannot both happen at different stages.  Hence our approximation for
$g_s(k)$ changes at most $2\cdot 2^{k+1}$ many times. 

 Hence $g_A \le_{wtt} K$, and it is straightforward to see that $A^*
\le_{wtt} g_A\oplus K\oplus A$ for any set $A$ because once we know
$g_A(k)$, we only need to query $K$ about strings below $g(k)$ (where
$g(k) > \left( A\rest g_A(k)\right)^*$ is recursive) to find $\left(
A\rest g_A(k)\right)^*$ and hence $A^*$. 
\end{proof}

\noindent
Theorem \ref{thm:r.e. degrees} now follows from Proposition \ref{prop:
r.e. anticomplex} and the techniques of Downey, Jockusch and Stob
\cite{Downey_Jockusch_Stob:_array_recursive_1,DJS99} and
Ishmukhametov \cite{Ishmu:99}. The fact that the array recursive r.e.\
wtt-degrees form an ideal now follows from Observation \ref{obs:ideal}.

\medskip
\noindent
One would perhaps hope that the previous result could be extended to
classes wider than the class of r.e.\ sets and their weak truth-table
degrees. Of course, if $A\le_{T(tu)} K$, then $A\le_{wtt} K$ and so $A$
is $\w$-r.e.; however, we now show that there are $\w$-r.e.\ sets $A$
which are anti-complex and yet $A \not\le_{T(tu)} K$. This shows that
the condition $B\le_T A\oplus K$ for the bound for $A$ with tiny use
cannot in general be improved to $B\le_{wtt} A\oplus K$.

We first need a lemma which again is not new, but which is not found in
standard references (an approximation, insufficient for our purposes,
is Theorem 9.14.6 in \cite{downeybook}). Let $\Omega$ be the halting
probability --- any left-r.e. Martin-L\"of random real would do.

\begin{lemma}\label{lem:Omega_compresses}
 For any r.e.\ set $A$, there is a reduction of $A$ to $\Omega$ with
use bounded below $2\log n$. 
\end{lemma}

\noindent
Indeed, we can even get a bound of $h(n)$ where $h$ is such that
$\sum_n 2^{-h(n)}$ is finite, such as $\log n + 2\log \log n$. 

\begin{proof}
 Let $\langle\Omega_s\rangle$ be an effective, increasing
approximation of $\Omega$ and, similarly, let $\langle A_s\rangle$ be
an effective enumeration of $A$. Let $h$ be a function such that
$\sum_n 2^{-h(n)}$ is finite. 

 If $n$ is the smallest number which enters $A$ at stage $s$, we
enumerate the interval $[\Omega_s, \Omega_s+2^{-h(n)}]$ into a Solovay
test $G$ which we enumerate. Since $n$ enters $A$ at most once, the
total measure of $G$ is at most $\sum_n 2^{-h(n)}$, which is finite by
assumption. 

 $\Omega$ is random, so it belongs to only finitely many of the
intervals in $G$. To compute $A(n)$ from $\Omega\rest h(n)$, find a
stage $t$ at which $\Omega_t \rest h(n) = \Omega\rest h(n)$; we claim
that $A(n)= A_t(n)$. If $n$ enters $A$ at a later stage $s$, then
$[\Omega_s,\Omega_s+2^{-h(n)}]$ is in $G$, but $\Omega- \Omega_t \le
2^{-h(n)}$ and $\Omega_t\le \Omega_s \le\Omega$, so we conclude that
$\Omega$ is in the interval $[\Omega_s,\Omega_s+2^{-h(n)}]$. Thus we
can get a wrong answer for only finitely many numbers $n$, and we can
find a reduction as required.
\end{proof}

\begin{prop} \label{prop:omega_differs}
 There is an anti-complex $\w$-r.e.\ set which is not reducible to $K$
with tiny use. 
\end{prop}

\noindent
Indeed, as the proof shows, there is such a set which is also the
difference of two left-r.e.\ reals. (We cannot get a left-r.e.\ real,
because every left-r.e.\ real is weak truth-table equivalent to an r.e.\ set.)

\begin{proof}
 By \cite[Theorem 4.1]{FS08}, there is a coinfinite r.e.\ set $A$ such
that every superset of $A$ is Schnorr trivial (indeed, any dense
simple set would do). Let $B= A\cup \Omega$. $B$ is Schnorr trivial
and thus anti-complex. $B$ is also $\w$-r.e., since it is a Boolean
combination of two sets which are wtt-reducible to $K$. 

 Now assume for a contradiction that $B\le_{T(tu)} K$. Then there is
some reduction $\Gamma^\Omega=B$ such that for all $n$,
$|\Gamma(\Omega\rest n)|>n$ because $\Omega$ has the same wtt-degree
as $K$. By Lemma \ref{lem:Omega_compresses}, there is a reduction
$\Delta^\Omega =A$ with the same property, as there is a reduction
from $A$ to $\Omega$ with use below $2\log n$. 

 We use the functionals $\Gamma$ and $\Delta$ to define a recursive
martingale which will succeed on $\Omega$, contradicting the fact that
$\Omega$ is random. The martingale $d$ is defined by induction on the
length of the binary strings which form its domain. We start with the
value 1. If $d(\s)$ is defined, we first calculate $\Delta(\s)(n)$ and
$\Gamma(\s)(n)$, where $n=|\s|$ (if either $|\Gamma(\s)|\le n$ or
$\Delta(\s)|\le n$, then we know that $\s$ cannot be an initial segment
of $\Omega$, so we can stop all betting). If $\Delta(\s)(n)=1$, then we
hedge our bets, that is, we let $d(\s0)=d(\s1)=d(\s)$. Otherwise, we put
all of the capital we have on the outcome $\Gamma(\s)(n)$, because in
this case, if $\s=\Omega\rest n$, then $A(n)=0$ and so
$B(n)=\Omega(n)$. Thus we let $d(\s i)=2d(\s)$ and $d(\s (1-i))=0$,
where $i= \Gamma(\s)(n)$. 

 Since $A$ is coinfinite, there are infinitely many $n$ at which we
double our money betting along $\Omega$, so $\lim_n d(\Omega\rest
n)=\infty$ as required for the contradiction. 
\end{proof} 

\section{Sets bounding nonrecursive sets with tiny use} \label{sec:high}

\noindent
We now turn to investigate the ranges of the relations $\le_{T(tu)}$
and $\le_{uT(tu)}$ (where the domain is restricted to the class of
nonrecursive sets to avoid triviality). Unlike their domains, these
ranges are not equal, because as we observed earlier, if
$A\le_{uT(tu)}B$ and $A$ is nonrecursive, then $B$ is high, whereas we
will shortly see that there are nonhigh sets which bound nonrecursive
sets with tiny use. First, we prove some results on the range of $\le_{uT(tu)}$.

\subsection{High degrees}
Unlike for Turing reducibility, with weak truth-table reducibility we
have to be careful when we deal with functions (elements of the Baire
space $\w^\w$) and sets (elements of the Cantor space $\{0,1\}^\w$). For
example, a function is always Turing equivalent to its graph, but if
it is not bounded by a recursive function, it may not be
wtt-equivalent to its graph. Our primary interest is to investigate
$\le_{T(tu)}$ on \emph{sets}, and so far we have not treated functions
as oracles in computations with recursive or tiny use. However, as a
technical tool, we can extend the definitions of $\le_{T(tu)}$ and
$\le_{uT(tu)}$ to include functions as oracles in the standard way;
weak truth table invariance still holds. In this context we have the
following result.

\begin{obs}\label{obs:dominant_uT(tu)}
 Let $G(f)$ be the graph of $f$. If $f$ is a dominant function, then
$G(f)\le_{uT(tu)} f$.
\end{obs}

\noindent
This allows us to characterise the range of $\le_{uT(tu)}$.

\begin{lemma}\label{lem:range_uT(tu)}
 Let $B\in \{0,1\}^\w$. There is some nonrecursive set $A$ such that
$A\le_{uT(tu)} B$ if and only if there is some dominant function $f\le_{wtt} B$.
\end{lemma}

\begin{proof}
In the proof of Proposition \ref{prop: high1} we noticed that if
there is some nonrecursive set $A$ such that $A\le_{uT(tu)} B$
witnessed by some reduction $\Phi$, then the map $n\mapsto
|\Phi(B\rest n)|$ is dominant and is weak truth-table reducible to
$B$. In the other direction, suppose that $f$ is dominant and that
$f\le_{wtt}B$. Let $A$ be the graph of $f$. Then $A\le_{uT(tu)} f$;
together with $f\le_{wtt} B$ we get $A\le_{uT(tu)} B$  from Observation
\ref{obs:wtt_invariance}. 
\end{proof}

\noindent
We know that every high Turing degree contains a dominating function,
but the weak truth-table degree of that function may not contain any set. 

\begin{lemma}\label{lem:sufficient_wtt}
 Let $f$ be a function such that $n\mapsto C(f(n))$ is bounded by some
recursive function. Then $f$ is wtt-equivalent to some set.
\end{lemma}

\begin{proof}
 Let $g$ be a recursive function which bounds $C(f(n))$. Let $A$ be
the set of pairs $(n,u)$ where $u$ is the first number below $g(n)$
which is discovered in some effective enumeration of the universal
machine $U$ to be mapped by $U$ to $f(n)$. Then $A\equiv_{wtt} f$. 
\end{proof}

\begin{prop}\label{prop:Martin_plus}
 Every high Turing degree contains a dominant function $\hat f$ such that
$C(\hat f(n))\le^+ n$.
\end{prop}

\begin{proof}
 Let $g$ be a dominant function; we first find an $f\le_T g$ with the
desired properties. 

 Once again, let $\langle \Omega_s\rangle$ be an effective increasing
approximation
of $\Omega$. Define $f$ by letting $f(n)$ be the least $s\le g(n)$
such that $\Omega_s \rest n = \Omega_{g(n)}\rest n$. It is certainly
true that $f\le_T g$.

 First we show that $C(f(n))\le^+ n$. Let $M$ be a machine that on an
input $\s$ of length $n$ outputs the least stage $s$ such that
$\s=\Omega_s\rest n$ if such a stage exists. Then for all $n$,
$M(\Omega_{g(n)}\rest n)= f(n)$, so $C_M(f(n))\le n$ as required. 

 Next, let $h$ be an order function. We first note that
$H(\Omega_{h(n)}\rest n)\le H(n)$ (as before, $H$ denotes \emph{prefix-free}
Kolmogorov complexity), and since $\Omega$ is random, for almost all
$n$, $\Omega_{h(n)}\rest n \ne \Omega\rest n$, as $H(\Omega\rest
n)\ge^+ n$. Thus we can let $\hat h(n)$ be the least $s>h(n)$ such
that $\Omega_s\rest n \ne \Omega_{h(n)}\rest n$; this too is a
recursive function, defined on almost every input. 

 Since $g$ is dominant, for almost all $n$, $g(n) > \hat h(n)$, which
implies that $\Omega_{g(n)}\rest n \ne \Omega_{h(n)} \rest n$ since
the approximation $\Omega_s\rest n$ does not return to old values,
and so $f(n) \ge \hat h(n) > h(n)$ for almost all $n$. Thus $f$ is dominant. 

\medskip
\noindent
 Next, we code a set $A$ in the Turing degree of $g$ into $f$ to get a
function which is Turing equivalent to $g$. We let $\hat f(n) = 2f(n)
+ A(n)$. Then $A\le_T \hat f$ and $\hat f \le_T f\oplus A \le_T g$, so
$\hat f\equiv_T g$. Since $\hat f$ bounds $f$, $\hat f$ is dominant,
and $C(\hat f(n))\le^+ C(f(n))\le^+ n$.
\end{proof}

\begin{cor}\label{cor:high_1}
 Every high Turing degree contains sets $A$ and $B$ such that $A\le_{uT(tu)} B$.
\end{cor}

\begin{proof}
 Let $\degr{a}$ be a high Turing degree. By Proposition
\ref{prop:Martin_plus} and Lemma \ref{lem:sufficient_wtt}, there is
some dominant $f\in \degr{a}$ which is wtt-equivalent to some set $B$, 
so of course $B\in \degr{a}$. By Lemma \ref{lem:range_uT(tu)}, there
is some set $A$ such that $A\le_{uT(tu)} B$. Indeed, we can take $A$
to be the graph of $f$. $A$ is thus Turing equivalent to $f$, so $A\in
\degr{a}$.
\end{proof}

\noindent
We can improve on the corollary in case, for example, the high degree
is also generalised low. 

\begin{theorem}\label{thm:high_2}
 If $\degr{a}$ is a Turing degree such that $\degr{a}\vee \degr{0}'
\ge_T \degr{0}''$, then for every $B\in \degr{a}$ there is some $A\in
\degr{a}$ such that $A\le_{uT(tu)} B$. 
\end{theorem}

\noindent
The point is that under the assumption that every $B\in \degr{a}$ is
wtt-equivalent to some dominant function $f$, we can let
$A$ be the graph of $f$. This means that we can prove the following
equivalent fact instead.

\begin{prop}\label{prop:high_2}
 If $\degr{a}$ is a Turing degree such that $\degr{a}\vee \degr{0}'
\ge_T \degr{0}''$, then every $B\in \degr{a}$ is wtt-equivalent to
some dominating function. 
\end{prop}

\begin{proof}
 Let $B\in \degr{a}$, and let $\seq{\vphi_e}$ be an enumeration of all
partial recursive
functions. Then $\Tot$, the collection of all indices $e$ such that
$\vphi_e$ is total, has Turing degree $\degr{0}''$, so there is some
Turing reduction $\Phi^{B\oplus K}= \Tot$. 

 We show that there is a set $E\subseteq \Tot$ which is recursively
enumerable in $B$ and such that for every (total) recursive function
$g$ there is some $e\in E$ such that $g=\vphi_e$. The set $E$ is
enumerated as follows: at stage $s$, if $\Phi^{B\oplus K_s}(e)\cvg
=1$, then we enumerate $g(e,\s)$ into $E$ with use $B\rest u$, where
$u$ is the use of the computation, $\s= K_s\rest u$ and the
instructions for calculating $\vphi_{g(e,\s)}$ are as follows. We
emulate $\vphi_e$ as long as $\s$ is an initial segment of $K_t$ for
stages $t\ge s$,
waiting for computations to converge, but if at some stage we observe
that $\s$ is no longer an initial segment of $K$, we make
$\vphi_{g(e,\s)}$ total by immediately converging on all inputs for
which we have not yet given an output and giving the answer 0. The
function $g$ is thus recursive. 

 Now $E$ is used to construct a dominant function $f\le_{wtt} B$: we
let $f(n)$ be the maximum of the values $\vphi_e(n)$ for the $e$ that
are enumerated into $E$ by stage $n$ with $B$-use at most $n$. 

 Again, we can modify $f$ to be a dominant function $\hat f
\equiv_{wtt} B$ by coding $B$ into $\hat f$, say again by letting
$\hat f(n) = 2f(n) + B(n)$.  
\end{proof}

\noindent
We do not know much in general about the range of $\le_{T(tu)}$. We
give some partial results in the following subsections. 

\subsection{Recursively enumerable degrees}
The techniques of the previous subsection can be improved to yield the
following result.

\begin{theorem}\label{thm:r.e._range}
 For every nonrecursive r.e.\ set $B$ there is a nonrecursive r.e.\
set $A$ such that $A\le_{T(tu)} B$.
\end{theorem}

\noindent
In order to prove this theorem, we take dominant functions which have
decent approximations. Let $\degr{h}$ be a high r.e.\ Turing degree.
By standard manipulations, we can get a dominant function $f\le_T
\degr{h}$ with an approximation with the following ``nice'' properties.
\begin{itemize}
 \item The approximation is increasing: for all $n$ and $s$,
$f_{s}(n)\ge f_{s-1}(n)$.
 \item If $f_s(n)\ne f_{s-1}(n)$, then $f_s(n)=s$.
 \item For all $s$ there is at most one $n$ such that $f_s(n)=s$ (this
is done by delaying changes in the approximation).
\end{itemize}

\begin{proof}[Proof of Theorem \ref{thm:r.e._range}]
 Let $B$ be a nonrecursive r.e.\ set. By a standard cone-avoiding
addition to the Sacks jump inversion theorem, there is a high r.e.\
degree $\degr{h}$ which does not compute $B$. Let $f\le_T \degr{h}$ be
an $\w$-r.e.\ dominant function with an approximation $\seq{f_s}$ as
described above. Let $g$ be a recursive bound on the mind-change
number of the approximation $\seq{f_s}$.

Enumerate a set $A$ as follows: for all $n\in B$, if $n$ enters $B$
at stage $s$, enumerate $f_s(n)$ into $A$. 

Suppose that we want to compute $A$ from $B$. To find out if $t\in A$,
we first go to stage $t$ and see if $f_t(n)=t$ for some $n\le t$ ---
if not, then $t$ is certainly not in $A$. If so, then $t$ is in
$A$ if and only if for the unique $n=n(t)$ such that $f_t(n)=t$, $n$
enters $B$ at a later stage $s$ before $f_s(n)$ changes. This gives a
reduction of $A$ to $B$ with identity use. 

In fact, $A\le_{T(tu)} B$. The idea is the following. Suppose again
that we want to find out whether $t$ is in $A$ and that we find that
$f_t(n)=t$. If we knew that $f(n)>t$, in other words, that there is a
later stage $s$ at which we have $f_s(n)\ne f_t(n)$, then we could
wait for that stage and see if $n$ entered $B$ before that stage or
not. Of course, we cannot always do this, because it may happen that
$t=f_t(n)=f(n)$ (or else $A$ would be recursive, whereas later we
show it is not). But now suppose that $h$ is an order function and
that we want to reduce $A$ to $B$ with use bounded by $h$. Then if
$n=n(t)<h(t)$, then we can consult $B(n)$ as before to compute $A(t)$.
If $h(t)\le n$, then $h^{-1}(n)\ge t$; since $f$ is dominant, $f(n)>t$
except for finitely many $n$, so we can employ the second tactic of
waiting for $f_s(n)$ to change in order to compute $A(t)$. In the
second case we do not consult $B$ at all, so overall we get a
reduction with use bounded by $h$. 

 Finally, to show that $A$ is not recursive, we see that $B\le_T
A\oplus f$ and recall that $B\not\le_T f$. To find $B(n)$, we
calculate the least stage $t$ at which $f_t(n)=f(n)$; if $n\notin
B_t$, then $n\in B$ if and only if $t\in A$. 
\end{proof}

\noindent
Just as we did for $\le_{tt}$, we can apply the ``tiny use'' operator
to many-one reducibility and say that $A\le_{m(tu)} B$ if for every
order function $h$ there is a recursive function $f$ dominated by $h$
such that $A= f^{-1}B$. The previous proof can be slightly modified to
show that for every nonrecursive r.e.\ set $B$, there are an r.e.\ set
$\hat B$ which is wtt-equivalent to $B$ and a nonrecursive r.e.\ set
$A$ such that $A\le_{m(tu)} \hat B$. We simply enumerate $(n,m)$ into
$\hat B$ if $n$ is enumerated into $B$ at stage $s$ and $f_s(n)$ is
the $m\tth$ value we see for $f_t(n)$ by stage $s$, that is, $m =
|\{f_t(n)\,:\, t\le s\}|$. Then, given $t$, we can find $n$ and $m$;
then $t\in A$ if and only if $(n,m)\in \hat B$ and, as described
above, this can be done with tiny use. To get $B\equiv_{wtt}\hat B$
rather than just Turing equivalence, we use a function $f$ which is
$\w$-r.e., the existence of which is guaranteed by the proof of
Proposition \ref{prop:Martin_plus}.

\medskip
\noindent
Theorem \ref{thm:r.e._range} cannot be extended to all $\Delta^0_2$
sets, as Downey, Ng and Solomon
\cite{DS06} constructed a $\Delta^0_2$ set which has minimal wtt-degree. 

Finally, there is an r.e.\ set $B$ which has minimal tt-degree
\cite{Marchenkov}. For such $B$, there can be no nonrecursive
$A\le_{tt(tu)}B$. Thus we cannot improve $\le_{T(tu)}$ in the theorem
to $\le_{tt(tu)}$, or, in the comments after the proof, get $\hat
B\equiv_{tt} B$ rather than $\hat B\equiv_{wtt} B$. 

\subsection{Hyperimmune-freeness}
Most hyperimmune-free Turing degrees do not contain any set $B$ for
which there is a nonrecursive $A$ such that $A\le_{T(tu)} B$. For
example, a Turing degree which is both minimal and hyperimmune free
does not contain such sets because if $\deg_T(X)$ is
hyperimmune free and minimal, then $\deg_{wtt}(X)$ is also minimal, as
every nonrecursive $Y\le_{wtt} X$ is Turing equivalent to $X$ and thus
also wtt-equivalent to $X$. Similarly, if $X$ is Martin-L\"of random
and $\deg_T(X)$ is hyperimmune free, then every nonrecursive
$Y\le_{wtt} X$ is truth-table equivalent to a Martin-L\"of random
set and so cannot be anti-complex, so again we get that every
$A\le_{T(tu)}X$ is recursive. 

Thus in the realm of the hyperimmune-free degrees, generic sets (in
the sense of either recursive Sacks forcing or forcing with sets of
positive measure) do not compute nonrecursive sets with tiny use. 

On the other hand, there is a hyperimmune-free $B$ with a
nonrecursive $A\le_{T(tu)} B$. This follows from the hyperimmune-free
basis theorem and the following theorem.

\begin{theorem}\label{thm:pi01range}
 There is a $\Pi^0_1$-class with no recursive elements consisting of
sets $B$ for which there are nonrecursive sets $A$ such that $A\le_{m(tu)}B$.
\end{theorem}

\noindent
(We remark that it is already known that there is a $\Pi^0_1$-class
with no recursive elements which consists of anti-complex sets; for
example, there is one which consists of Schnorr trivial sets
\cite{FS08} and one which consists of sets $A$ such that $\deg_T(A)$
is r.e.\ traceable.)

\begin{proof}
 We imitate part of the the proof of Theorem \ref{thm:r.e._range}.
Again, let $\seq{f_s}$ be an $\w$-r.e.\ approximation for a dominant
function $f$ with the properties discussed above; say $g$ is a
recursive function which bounds the number of possible values
$m(n)=|\{f_{s}(n) : s<\w\} |$.

 For $n<\w$ and $k\le m(n)$, let $\pi(n,k)$ be the $k\tth$ value of
$f_s(n)$. Thus $\pi(n,m(n))=f(n)$. Now let $D = \{(n,k)\,:\, k<m(n)\}$
and $E =  \pi[D]$; both are r.e.\ sets. Furthermore, $E$ is nonrecursive
as $f\le_T D\le_T E$: for each $n$, we find $\pi(n,k)$ recursively in
$k$ and so, consulting $E$, determine if $k<m(n)$ or not. 

 We can thus split $E$ into a pair $E_0$ and $E_1$ of recursively
inseparable r.e.\ sets. Let $\mathcal P$ be a $\Pi^0_1$-class of sets
which separate $D_0 = \pi^{-1}E_0$ from $D_1 = \pi^{-1} E_1$ (sets
that contain $D_0$ and are disjoint from $D_1$). Let $B$ be any
element of $\mathcal P$ and let $A = \pi[B]$. Then $A$ separates
$E_0$ and $E_1$, so $A$ is not recursive. We claim that
$A\le_{T(tu)}B$, indeed that $A\le_{m(tu)} B$. 

 The argument is similar to that of the proof of Theorem
\ref{thm:r.e._range}. For any $t$, $t\in A$ if and only if $t$ is the
$k\tth$ value of $f_s(n)$ for some (unique) $n$ which can be
effectively obtained from $t$; the question is whether $(n,k)\in B$.
As before, if $h$ is any order function, then for all but finitely
many $t$, either $(n,k)<h(t)$ or $k<m(n)$. In the latter case,
$(n,k)\in D$, so we do not need to consult $B$ about $(n,k)$, as the
value $B(n,k)$ is decided by which one of $D_0$ or $D_1$ the pair
$(n,k)$ is enumerated into, a fact which is revealed to us with
sufficient patience.  
\end{proof}

% 
% \bibliographystyle{plain}
% \bibliography{../../biblist}
% \end{document}
% 

\end{document}